\documentclass[12pt]{article}
\usepackage[mathscr]{eucal}
\usepackage{amsthm,amsmath,amssymb,amscd}
\usepackage{graphicx}
\graphicspath{{figures/}}
\usepackage{latexsym,enumerate}
\usepackage{subfig}
\usepackage{color}
\usepackage{pgf,tikz}
\usetikzlibrary{angles}
\usetikzlibrary{arrows}
\usepackage[outline]{contour}
\usepackage{caption}
\newtheorem{theorem}{Theorem}

\newtheorem{proposition}{Proposition}

\theoremstyle{definition}
\newtheorem{definition}{Definition}

\newtheorem*{remark}{Remark}

\begin{document}
\title{Some Polynomial Conditions for Cyclic Quadrilaterals, Tilted Kites and Other Quadrilaterals}
\author{Manuele Santoprete\thanks{ Department of Mathematics, Wilfrid Laurier
University E-mail: msantopr@wlu.ca. ORCID: 0000-0003-0501-7517}} 
 \maketitle
\begin{abstract}
    In this paper, we investigate some polynomial conditions that arise from Euclidean geometry. First we study polynomials related to quadrilaterals with supplementary angles, this includes convex cyclic quadrilaterals, as well as certain concave quadrilaterals. Then we consider polynomials associated with quadrilaterals with some equal angles, which include convex and concave tilted kites. Some of the results are proved using Groebner bases computations.  
  \end{abstract}

\renewcommand{\thefootnote}{\alph{footnote})}


\section{Introduction}
This paper studies some polynomial equations that arise from Euclidean Geometry.
The variables appearing in these polynomials are the mutual distances between four points, and the equations in question  express some properties of quadrilaterals. 
In \cite{10.1007/978-3-642-21898-9_34} Pech studied several polynomials, three of  them, which we denote $ P $ , $ S $ and $ K $ (see Section \ref{sec:cyclic} for their expressions), are necessary conditions for a quadrilateral to be cyclic. The condition $ P=  0$  is the equation  that appears in the celebrated Ptolemy's Theorem, while $ S = 0 $ and $ K = 0  $ are a cubic and a quartic equations, respectively. The condition $ K = 0 $ follows from the cosine law, while  $ S = 0 $ is less well known. All three conditions are closely connected, in fact, Pech proved that $ P = 0 $ is equivalent to $( K = 0 \wedge S = 0)$.   \cite{10.1007/978-3-642-21898-9_34}. In this paper we give a simpler proof of this fact that doesn't use algebraic geometry. 

The  polynomials studied by Pech  \cite{10.1007/978-3-642-21898-9_34}  have been useful in studying cyclic central configurations of four bodies in Celestial Mechanics \cite{cors2012four,santoprete2021cocircular}. Similar polynomials, related to another type of quadrilaterals, namely trapezoids,  have also been found to be useful in studying central configurations of four bodies 
\cite{santoprete2019planarity,santoprete2021trapezoidal}.  In light of this, finding additional polynomial conditions characterizing various configurations of four points is not only interesting from a geometry standpoint, but also from the   perspective of potential applications to Celestial Mechanics.

In this paper we study the polynomial $ R $ (see Section \ref{sec:cyclic}), which also appeared in Pech's paper  \cite{10.1007/978-3-642-21898-9_34}, but wasn't analyzed in detail. We show that the condition $ R = 0 $ selects certain configurations (either convex or concave) that possess supplementary angles, but are not cyclic.  

We also study  other polynomials that are important for quadrilateral where some of the angles are equal. Recall that a kite is a quadrilateral with two pairs of equal-length adjacent sides.
Kites can be either concave or convex, and they have equal opposite angles. A way to generalize kites is to consider quadrilaterals with equal opposite angles. These quadrilaterals are called tilted kites \cite{graumann2005investigating,josefsson2018properties} or angle quadrilaterals \cite{de2009some}, and they can be convex or concave. The name tilted kite seems more descriptive since such a quadrilateral can be obtained from a kite by skewing  or tilting one of the angles.
The new polynomials we study in this paper are useful in characterizing concave and convex titled kites. 

The paper is organized as follows. In Section 2 we prove several useful results from Euclidean geometry and algebraic geometry. In Section 3 we study some polynomials related to cyclic quadrilaterals and other geometrical figures with supplementary angles. We start with analyzing conditions that were already studied by Petch \cite{10.1007/978-3-642-21898-9_34}. We give a more precise statement of the Converse of Ptolemy's theorem and a simpler proof of one of the results by Pech. We conclude with an in detail analysis of the condition $ R = 0 $ (see equation \eqref{thm:R=0}) and its geometrical meaning. 
In Section 4 we study  some new polynomial  conditions related to tilted kites and other geometrical figures with two equal angles. 
In Section 5, we conclude by introducing other sets of similar polynomials, and giving some ideas on how to find additional interesting polynomials associated with quadrilaterals. 
\section{Preliminaries}

Consider four points in $ \mathbb{R}  ^3 $ given by $ A $, $ B  $, $ C  $, and $ D =  $ (here we are using a right-handed coordinate system), and let $ a = |AB| $,  $ b = |BC| $, $ c = |CD|$, $ d = |DA| $, $ e = |AC| $, $ f = |BD| $, be the Euclidean distances between pair of points. 

The following theorem gives a conditions for four points to lie on a plane. 
\begin{theorem} 
If the  four points  $ A, B, C , D \in \mathbb{R}  ^3 $ with distances $ a , b ,c ,d, e $  and $ f $,   belong to a $ 2 D $  affine subspace of  $ \mathbb{R}  ^3 $, then 
\[   CM = \begin{vmatrix}
            0 & 1 & 1 & 1 & 1 \\
            1 & 0 & a^2 & e^2 & d^2  \\
            1 & a^2 & 0 & b^2 &f^2 \\
            1 & e^2 & b^2 & 0 & c^2\\
            1 & d^2 & f^2 & c^2 & 0
        \end{vmatrix}= 0.
\]
\end{theorem} 
The determinant  $ CM$ is the Cayley-Menger determinant of the four points $ A,B,C $ and $ D $, and, in general  $ CM = 288 V ^2 $, where  $ V $ is the volume of the 3-simplex of edge lengths $ a,b,c,d,e,f $. 
The condition $ CM = 0 $ is known in Euclidean geometry as Euler's four point relation. 

Another useful result is Euler's quadrilateral theorem:
\begin{theorem}[Euler's Quadrilateral Theorem] 
Consider a convex quadrilateral with consecutive sides $ a,b,c,d $ and diagonals $ e, f $. Then 
\[ a ^2 + b ^2 + c ^2 + d ^2 - e ^2 - f ^2 = 4 v ^2 \]
where $ v $ is the distance between the midpoints of the diagonals $ e $ and $ f $. 
\end{theorem} 
This theorem can be generalized so that it applies  to four arbitrary points in $ \mathbb{R}  ^3 $ \cite{kandall2002classroom,knill2018some}.
\begin{theorem}[Generalized Quadrilateral Theorem]\label{thm:generalized-quadrilateral} 
    Suppose $ A, B, C , D $ are four points in $ \mathbb{R}  ^3 $ forming a tetrahedron of sides $ a = |AB| $, $ b = |BC| $ , $ c = |CD| $, $ d = |AD| $, $ e = |AC| $ and $ f = |BD| $. Then we have an identity for each pair of edges with no common vertex. More specifically we have:
     \begin{align*} 
        a ^2 + b ^2 + c ^2 + d ^2 - e ^2 - f ^2 & = 4 v _1 ^2 \\
        -a ^2 + b ^2 - c ^2 + d ^2 + e ^2 + f ^2 & = 4 v _2 ^2 \\
        a ^2 -b ^2 + c ^2 - d ^2 + e ^2 + f ^2 & = 4 v _3 ^2 \\
    \end{align*} 
where $ v _1 $ is the distance between the midpoints of the segments $AC $ and $BD$, $ v _2 $ is the distance between the midpoints of the segments $AB$ and $ CD$, and  $ v _3 $ is the distance between the midpoints of the segments $BC$ and $AD$.  
\end{theorem}
\begin{proof}
    These formulas can be verified by writing everything in coordinates and expanding out.  
\end{proof}
Note that, in particular, this theorem applies to concave and convex quadrilateral.

We now restrict our attention to  the case the four points belong to a $ 2D $ affine subspace of  $ \mathbb{R}   ^3 $. In this case, we  can write the four points as:
 $ A = (A _x , A _y) $, $ B = (B _x , B _y) $, $ C = (C _x , C _y) $, and $ D = (D _x , D _y) $.
 From now on we also assume that the four points are distinct so that $a,b,c,d,e,f >0$.

A {\bf concave configuration} of four points  has  one point which is located strictly inside the convex hull of
the other three, whereas a {\bf convex configuration} does not have a point  contained in the
convex hull of the other three points. We say that a configuration is degenerate if three or more points lie on the same straight line.
For a convex configuration we  say that the points are {\it ordered sequentially } if they are labelled consecutively while traversing the boundary of the quadrilateral (clockwise or counterclockwise). 

In this paper we will use the following notation.  When we say that the convex hull of the points $ A, B, C , D $ is 
  $ ACD $ we mean that  the configuration is concave with the point $ B $  in the interior of the triangle $ ACD $,
 and $ ACD $ is the order of the points  when traversing the boundary counterclockwise. 
  When we say that the convex hull is $ ADBC $ we  mean that  the configuration is convex and $ ADBC $ is be the order of the points when traversing the boundary counterclockwise. 


Let 
\[\Delta _{ ABC } = \begin{vmatrix}
     A_x & A_y & 1 \\
     B_x & B_y & 1 \\
     C_x & C_y & 1
\end{vmatrix}, \quad  \Delta _{ ABD } = \begin{vmatrix}
     A_x & A_y & 1 \\
     B_x & B_y & 1 \\
     D_x & D_y & 1
\end{vmatrix}, 
\]
and 
\[\Delta _{ BCD } = \begin{vmatrix}
     B_x & B_y & 1 \\
     C_x & C_y & 1 \\
     D_x & D_y & 1
\end{vmatrix}, \quad  \Delta _{ ACD } = \begin{vmatrix}
     A_x & A_y & 1 \\
     C_x & C_y & 1 \\
     D_x & D_y & 1
\end{vmatrix}, 
\]
where $ \Delta _{ ABC } $ denotes twice  the signed area of the triangle $ ABC $. 
Signed means that the area is positive if the vertices of the triangle are ordered counterclockwise, negative if the vertices are ordered clockwise, and zero if the points are collinear.
The following theorem  is helpful in determining the convex hull of four points in $ \mathbb{R}  ^2 $ and 
in classifying configurations as concave or convex. \begin{theorem}\label{thm:hulls} Let $ A, B, C  $ and $ D $ be points on $ \mathbb{R}  ^2 $ as defined above, and let $N=\Delta _{ ABC } \cdot \Delta _{ ACD } $ and $ M = \Delta _{ ABD } \cdot \Delta _{ BCD } $. The configuration formed by the points $ ABCD $ is convex if  and only if 
    \[ N \geq 0 \mbox{ and } M \geq 0 \] 
    or \[ N \leq 0 \mbox{ and }  M \leq 0. \]
The configuration is convex and ordered sequentially if and only if  
 \[ N \geq 0 \mbox{ and } M \geq 0. \] 
Furthermore, if any of the $\Delta$'s is zero, the three points are collinear. If all the $ \Delta $'s are zero, then all four points are collinear and the hull will be either a line or a point. If all the $ \Delta $'s are non-zero and $ \Delta _{ ABC } >0 $ then the convex hull is given in the following table 
\begin{center}
\begin{tabular}{ |c|c|c|c|c| } 
 \hline
 $ABC$ & $ABD$ & $BCD$ & $ACD$ & Convex Hull \\
 \hline 
$ +$ & $+$ & $+$ &$+$ & $ABCD$ \\
$ +$ & $+$ & $-$ &$-$ & $ABDC$ \\ 
$ +$ & $-$ & $+$ &$-$ & $ADBC$ \\ 
$ +$ & $+$ & $+$ &$-$ & $ABC$ \\
$ +$ & $+$ & $-$ &$+$ & $ABD$ \\ 
$ +$ & $-$ & $+$ &$+$ & $BCD$ \\ 
$ +$ & $-$ & $-$ &$-$ & $CAD$ \\
$ +$ & $-$ & $-$ &$+$ &  Not realizable\\
 \hline
\end{tabular}
\end{center}
If $ \Delta _{ ABC }<0 $ we obtain an analogous table:
\begin{center}
\begin{tabular}{ |c|c|c|c|c| } 
 \hline
 $ABC$ & $ABD$ & $BCD$ & $ACD$ & Convex Hull \\
 \hline 
$ -$ & $-$ & $-$ &$-$ & $ADCB$ \\
$ -$ & $-$ & $+$ &$+$ & $ACDB$ \\ 
$ -$ & $+$ & $-$ &$+$ & $ACBD$ \\ 
$ -$ & $-$ & $-$ &$+$ & $ACB$ \\
$ -$ & $-$ & $+$ &$-$ & $ADB$ \\ 
$ -$ & $+$ & $-$ &$-$ & $BDC$ \\ 
$ -$ & $+$ & $+$ &$+$ & $CDA$ \\
$ -$ & $+$ & $+$ &$-$ &  Not realizable\\
 \hline
\end{tabular}
\end{center}
\end{theorem}
Although this result is fairly standard, the only reference to this theorem I am aware of is contained in this stackoverflow post \cite{comingstorm2018convex}.

The next theorem gives a criterion for four points with given coordinates to lie on  a  circle:
\begin{theorem}\label{thm:cyclic}
   Let    $ A = (A _x , A _y) $, $ B = (B _x , B _y) $, $ C = (C _x , C _y) $, and $ D = (D _x , D _y) $
be points on $ \mathbb{R}  ^2 $ then 
\begin{equation}\label{eqn:det}
    \mathcal{C} =  \begin{vmatrix}
    A _x ^2 + A _y ^2 &  A_x & A_y & 1 \\
     B _x ^2 + B _y ^2 & B_x & B_y & 1 \\
     C _x ^2 + C _y ^2 & C_x & C_y & 1\\
     D _x ^2 + D _y ^2 & D_x & D_y & 1 
\end{vmatrix}=0,
\end{equation} 
if and only if the four points lie on a circle or on a straight line. 

\end{theorem}
\begin{proof} 
    Assume equation \eqref{eqn:det} is satisfied. 
The  expression
\[  \begin{vmatrix}
     x ^2 + y ^2 &  x & y & 1 \\
     B _x ^2 + B _y ^2 & B_x & B_y & 1 \\
     C _x ^2 + C _y ^2 & C_x & C_y & 1\\
     D _x ^2 + D _y ^2 & D_x & D_y & 1 
\end{vmatrix}=0,
\]
can be written as 
\begin{equation}\label{eqn:circle-line} 
    \alpha (x ^2 + y ^2) + \beta x + \gamma y + \delta = 0, 
\end{equation} 
where  $ \alpha = \Delta _{ BCD } $.  It is evident from equation \eqref{eqn:circle-line}  that when $ \alpha \neq 0 $  the points $ A,B,C,D $ lie on a circle, and when  $ \alpha  = 0 $ the points $ A, B,C $ and $ D $ lie on a straight line. 

\end{proof} 

We will now recall a few facts from algebraic geometry, see \cite{cox2013ideals} for more details. 
 \begin{definition} 
     Let $k$ be an arbitrary field and  $ I \subset k[x _1 , \ldots x _n ] $  an ideal.  The {\bf radical } of $I$, denoted by $ \sqrt{ I } $, is the set 
     \[
         \{ f : f ^m \in I \mbox{ for some integer } m \geq 1 \}.
     \]
 \end{definition} 
We now give a theorem to test whether $ f \in \sqrt{ I } $:
\begin{proposition}[Radical Membership]\label{prop:RadicalMembership}  
    Let $ k $ be an arbitrary field and let $ I = \left\langle f _1,\ldots, f _s \right\rangle \subset k[x _1 , \ldots , x _n ] $ be an ideal. Then $ f \in \sqrt{ I } $ if and only if the constant polynomial $ 1 $ belongs to the ideal $ \tilde I = \left\langle f _1 , \ldots , f _s , 1 - y f \right\rangle \subset k [ x _1 , \ldots x _n , y ] $.
\end{proposition}
From this  it follows that to determine whether or not  $ f \in \sqrt{ \left\langle f _1 , \ldots , f _s \right\rangle } $ it is enough to compute a reduced Groebner basis of the ideal $  \left\langle f _1 , \ldots , f _s , 1 - y f \right\rangle $ with respect to some ordering. If the result is $ \{ 1 \} $, then $ f \in \sqrt{ I } $, otherwise $ f \neq \sqrt{ I } $. Note that the  Maple \texttt{Basis} command  always computes a reduced Grobner bases, so it is easy to use this approach.  
\section{Some Polynomial Expressions and Cyclic Quadrilaterals}\label{sec:cyclic}
Pech \cite{10.1007/978-3-642-21898-9_34} considered the following polynomials and their relation with cyclic quadrilaterals. 
\begin{align} P & = ac + b d - ef\\
               Q & = ac + bd + ef\\
               S & = e(ab+cd)-f(ad+bc)\\
               K & = e ^2 (ab + cd) - (a ^2 + b ^2) cd - (c ^2 + d ^2) ab\label{eqn:K}\\
               R & = (bc + ad) ^2 - e ^2 (a ^2 + b ^2 + c ^2 + d ^2 - e ^2 - f ^2) \\
               W & =ac(-a^2-c^2+b^2+d^2+e^2+f^2) \nonumber\\
               & +bd(a^2+c^2-b^2-d^2+e^2+f^2) -ef(a^2+c^2+b^2+d^2-e^2-f^2)
\end{align}
It is not difficult to check that the polynomials above verify the following identities
\begin{align}
    -PW+S^2+CM/2& = 0 \label{eqn:identity1}\\ 
    -eS+K+(bc+ad)P& =0 \label{eqn:identity2}\\
    2(K^2-PQR)+e^2CM & =0,\label{eqn:identity3}
\end{align} 
where $CM=288V^2 $ denotes  the Cayley-Menger determinant.  The first and third equations above provide interesting way to organize the terms of the Cayley-Menger determinant.  The first of these identities has been useful in studying cyclic central configurations in Celestial Mechanics \cite{cors2012four,santoprete2021cocircular}. 


We begin by  stating  without proof the well known Ptolemy's theorem:
\begin{theorem}[Ptolemy's Theorem]\label{thm:Ptolemy}
   Let $ A, B, C ,D $ be four points lying in counterclockwise  or clockwise order  on a circle, then $ P = 0 $. 
 The same conclusion holds if the 4 points lie in order on a straight line.
\end{theorem}
Next we state  and prove of  the so called Converse of Ptolemy's Theorem, following the approach of Pech  \cite{10.1007/978-3-642-21898-9_34}. We include this theorem and its proof because we want to include the case of a straight line which was not explicitly considered in Pech's paper. 
\begin{theorem}[Converse of Ptolemy's Theorem]\label{thm:Converse-Ptolemy} 
    Let $ A, B , C $ and $ D $ be four points in the plane that satisfy  $ P = 0 $ then they  form a convex  configuration (or quadrilateral) and lie on a circle or a straight line. If they lie on a circle their counterclockwise order is either $ ABCD $ or $ ADCB $. 
\end{theorem} 

\begin{proof} 
Without loss of generality we choose $ A = (0, 0) $, $ B = (a, 0) $ , $ C = (u, v) $, and $ D = (w, z) $, as in \cite{10.1007/978-3-642-21898-9_34}. With this assumption the  expression for $ \mathcal{C} $ given in equation \=ref{eqn:det} can be written as  
\[\mathcal{C} = a (- a v w + v w ^2 + a u z - u ^2 z - v ^2 z - v ^2 z - v ^2 z + vz ^2) = 0. \]
We also define the following polynomials
\begin{align*} 
f _1 & = (u-a)^2 + v ^2 - b ^2 \\
f _2 & = (w-u) ^2 + (z-v) ^2 - c ^2\\
f _3 & = w^2 + z ^2 - d ^2 \\
f _4 & = u ^2 + v ^2 - e ^2 \\
f _5 & = (w-a)^2+z^2-f^2
\end{align*}
which describe the quadrilateral algebraically. 

Suppose  that $ P = 0 $ and consider the ideal $ I = \left\langle f _1 , f _2 , f _3 , f _4,f_5\right\rangle  $.
A reduced Groebner basis for the ideal $  
\left\langle f _1 , \ldots , f _5 , P , 1 - t \mathcal{C} \right\rangle $ is $\{  1 \} $. Therefore, by  Proposition \ref{prop:RadicalMembership} we have that $ \mathcal{C} \in \sqrt{ I \cup \{ P \} } $.
This shows that  $ P = f _1 = f _2 = f _3 = f _4 = f _5 = 0 $ implies $ 
\mathcal{C} = 0 $. It follows by Theorem \ref{thm:cyclic} that the points either lie on a circle or on a straight line. 

We now prove convexity. Consider the ideal $ K = I \cup \{P\} \cup \{N - t\}$, where $ t $ is a slack variable and $ N $ is defined in Theorem \ref{thm:hulls}. Eliminating the variables $ e, f, u , v , w, z $ in $ K $ it is possible to show  that 
\[
    N = \frac{ abc d (- a + b + c + d) (a - b + c + d) (a + b - c + d) (a + b  + c - d) } { 4 (ab + cd) ^2 }.       
\]
A similar computation  shows that
\[
    M = \frac{ abc d (- a + b + c + d) (a - b + c + d) (a + b - c + d) (a + b  + c - d) } { 4 (bc + ad) ^2 }.       
\]
From the triangle inequality it follows that $ (- a + b + c + d) \geq 0 $, $ (a - b + c + d) \geq 0 $, etc., where equality means that the four points lie on  the same line. This shows that the configuration is convex. Moreover, if the points do not lie on the same line then $ M , N >0 $, and all the triangles $ \Delta _{ ABC } $, etc. are non zero. Thus, by Theorem \ref{thm:hulls} it follows that their counterclockwise order must be either $ ABCD  $ or $ ADBC $. See Figure \ref{fig:cyclic} for an example of a cyclic quadrilateral with counterclockwise order $ABCD$. 

\end{proof}
\begin{remark}
   Of course there are two other polynomials similar to $ P $, these are $ P _T $ and $ Q _T  $ and are defined in equations \eqref{eqn:PT} and \eqref{eqn:QT} respectively.
   $ P _T $ appears when Ptolemy's theorem is stated for quadrilaterals with diagonals $ b $ and $ d $. 
$ Q _T $ appears when Ptolemy's theorem is stated for quadrilaterals with diagonals $ a $ and $ c $. These polynomials will play an important role later on in this paper. 

\end{remark} 
We now state a prove a theorem that shows how the conditions $ P = 0 $ , $ K = 0 $ and $ S = 0 $ are related. This theorem was proved in  Pech \cite{10.1007/978-3-642-21898-9_34}, but here we give a very simple proof that does not use tools of algebraic geometry. 
\begin{theorem}
 Let $ A, B , C $ and $ D $ be four points in the plane then 
 \[
     (K = 0 \wedge S = 0) \Leftrightarrow P = 0. 
 \]
\end{theorem} 
\begin{proof} 
    Suppose $ P = 0 $. Since the configuration is planar then $ CM = 0 $. Thus,  equation \eqref{eqn:identity1} gives $ S  = 0 $. Since $ S = P = 0 $ equation \eqref{eqn:identity2} also gives $ K = 0 $. 
    Now assume that $ K = 0 $ and $ S = 0 $. Then, from equation \eqref{eqn:identity2} it follows that $ ( b c + a d) P = 0 $, but since $ a, b, c, d>0 $ this means that $ P = 0 $.  
\end{proof} 
The next theorem shows that the condition $ R = 0 $ is somewhat complementary to $ P = 0 $. In fact,  while the figure obtained by imposing  $ R = 0 $ still has two supplementary angles, it is quite different from a cyclic quadrilateral. 
A convex example of such configuration  is given in  Figure \ref{fig:supplementary}, while a concave example is given in Figure \ref{fig:supplementary-concave}. Note that in both these examples these configurations can be obtained by folding a cyclic quadrilateral about one of the diagonals. The following theorem appears to be new.  
\begin{figure}[!tbp]
  \centering
  \subfloat[]{\includegraphics[width=0.35\textwidth]{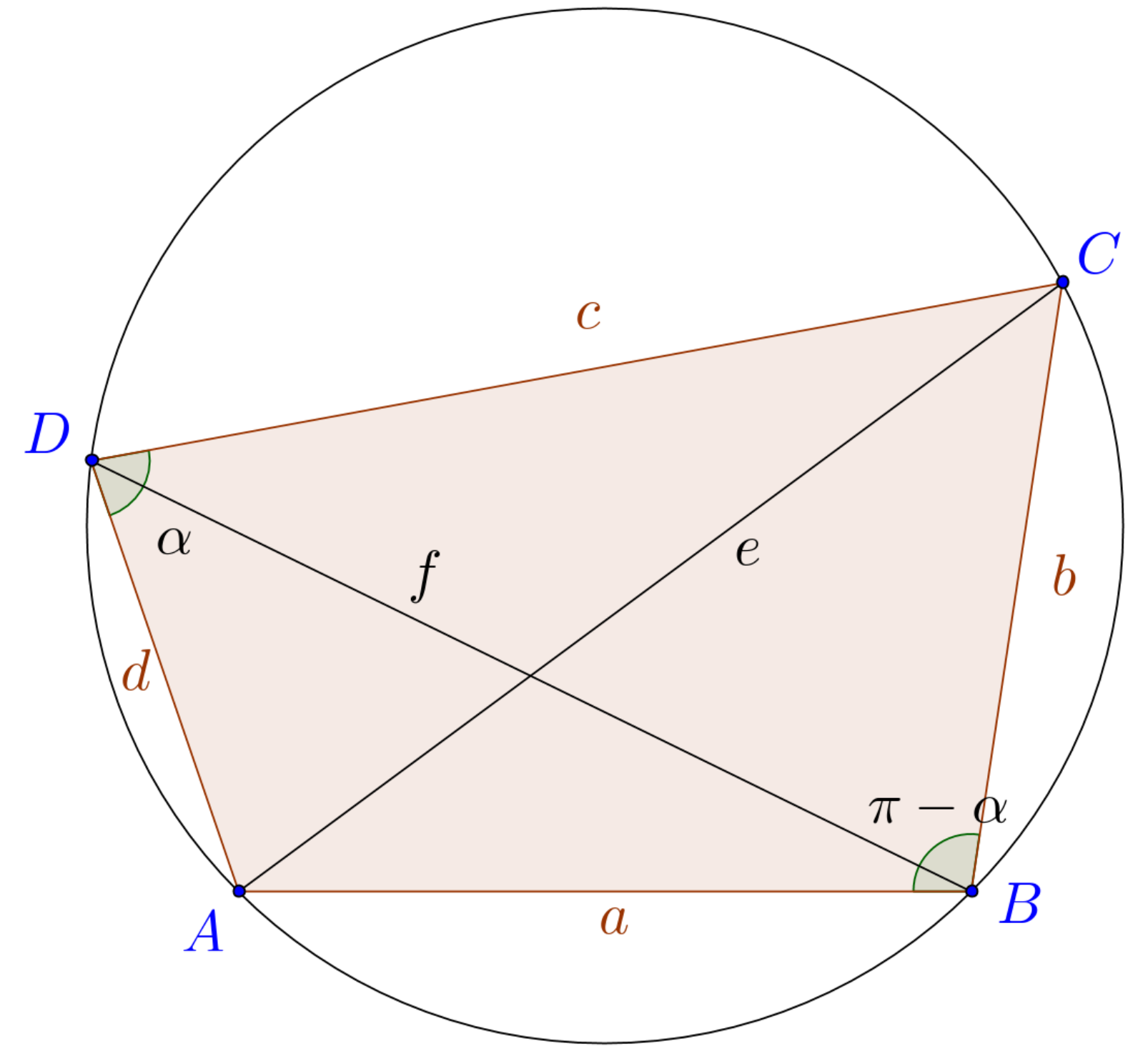}\label{fig:cyclic}}
  \hfill
  \subfloat[]{\includegraphics[width=0.6\textwidth]{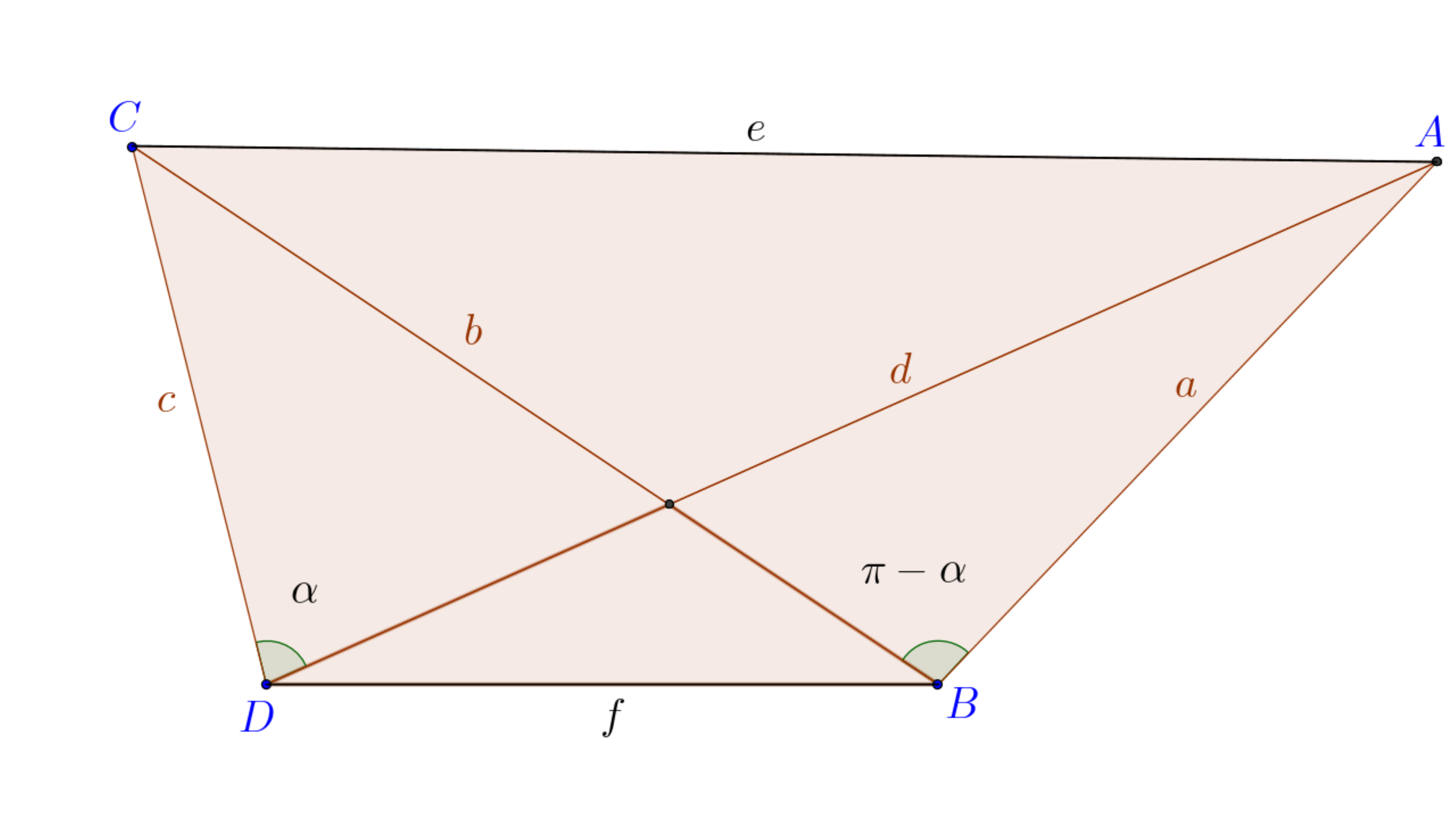}\label{fig:supplementary}}
  \caption{(a)  A cyclic quadrilateral with supplementary angles $\angle CDA $ and $ \angle CBA$. (b) A quadrilateral  $ ACDB $ with  supplementary angles $\angle CDA $ and $ \angle CBA$.}
\end{figure}

\begin{theorem}\label{thm:R=0}
     Let $ A, B , C $ and $ D $ be four points in the plane that satisfy  $ R = 0 $ then the angles  $\angle CDA  $ and $\angle CBA  $ are supplementary.  
     Moreover, if  all the $ \Delta  $'s are nonzero then 
     \begin{enumerate}
         \item If  $ \Delta _{ ABC } <0 $ the convex hull  of the points $ A, B , C, $ and $ D $ is one of $ ACB $,  $CDA$, $ ACDB $, and $ ACBD $. 
         \item If  $ \Delta _{ ABC }  >0 $, then  the convex hull is one of  $ABC$, $CAD$, $ABDC$, and $ ADBC $. 
     \end{enumerate}  
    Finally, If at least one of the $ \Delta  $'s is zero then we have the following degenerate cases
     \begin{enumerate}
         \item The points $ A,B,C $ and $ D $ are collinear.  
         \item The points $ A, B  $ and $ D  $ are collinear and $ b = c $, see Figure \ref{fig:degenerate-case-1}. 
         \item The points $ B , C $ and $ D $ are collinear and $ a = d $ , see Figure \ref{fig:degenerate-case-2}.   
     \end{enumerate}  

\end{theorem} 
\begin{proof}
   Assume $ R = 0 $.  Since the points are on the plane we also have that $ CM = 0 $. Hence equation \eqref{eqn:identity3} gives  that $ K = 0 $.
The equation $ K = 0 $ can also be written as 
\[
    ab \left[ e ^2 - (c ^2 + d ^2) \right] = - cd \left[ e ^2 - (a ^2 + b ^2) \right].   
\]
Applying the law of cosines to the triangles $ ACD $  and $ ABC$ we have $ e ^2 = c ^2 + d ^2 - 2 cd \cos \alpha $ and $ e ^2 = a ^2 + b ^2 + 2ab \cos \beta $, where $\measuredangle CDA = \alpha $ and $ \measuredangle CBA = \beta $. Hence we obtain the equations 
\[ ab [ e ^2 - (c ^2 + d ^2)  ] = - 2 abcd \cos \alpha, \quad cd[ e ^2 - (a ^2 + b ^2) ] = - 2abcd \cos \beta . \]
Hence, $ K = 0 $ implies that $ \cos \alpha = - \cos \beta $, which can be also written as 
\[
    \cos \alpha + \cos \beta = 2 \cos \left( \frac{ \alpha + \beta } { 2 } \right) \cos \left( \frac{ \alpha - \beta } { 2 } \right)=0,
\]
which means that $  \alpha + \beta = 2\pi n - \pi $ or $\alpha - \beta  = 2 \pi m - \pi $. Since $ \alpha , \beta \leq \pi $ the only possible solution  is $ \alpha  + \beta = \pi $.  This proves that the angles  $\angle CDA  $ and $\angle CBA  $ are supplementary.

Choose a  coordinate system such that $ A=(u,v) $, $ B = (f,0) $, $ C = (w, z) $ and  $ D = (0,0) $. Then we have the following polynomials
\begin{align*} 
h _1 & = (u - f) ^2 + v ^2 - a ^2 \\
h _2 & = (u - w) ^2 + (z - v) ^2 - e ^2\\
h _3 & = w ^2 + z ^2 - c ^2 \\
h _4 & = u ^2 + v ^2 - d ^2 \\
h _5 & = (w - f) ^2 + z ^2 - b ^2 \\ 
\end{align*} 
We show that $ N \leq 0 $. Consider the ideal $ J_N=<h _1, h _2, h _3 , h _4, h _5 , R, N - t>$, where $ t $ is a slack variable. 
Calculating the elimination ideal of $ J $  with Maple,  eliminating the variables $ e,f, u,v, w, z $ yields a polynomial $ L $  in $ a,b,c,d $ and $ t $. Solving  $ L = 0 $ for $t$ yields
\[
    t = -\frac{ (-d+a+c+b)(d-a+c+b)(d+a-c+b)(d+a+c-b)abcd } { 4(ab+cd)^2}.
\]
From this expression for $ t $ it follows that $ N = t \leq 0 $, since in a quadrilateral  $ -d+a+c+b \geq 0 $, $ d-a+c+b \geq 0 $, etc. (which follows from the  triangle inequality).
Performing a similar computation it is possible to find an expression for $ M $, but it turns out that $ M $ can also take positive values, and so the computation is not useful for this proof. 

We now study the degenerate cases where three or more points lie on the same line. 
Note that the equality   $ -d+a+c+b =0 $, or one of the other similar ones, means that $ A,B, C $ and $ D $ are collinear. As a consequence,  $N=t=0$  if and only if $ A,B,C $ and $ D $ are collinear. This means that  the triangles $ ACD $ and $ ABC$  have nonzero areas, unless the quadrilateral degenerates to a line. Therefore, the condition $ R = 0 $ implies that  $ A,  C $ and  $  D  $ cannot be collinear, and $ A , B $ and $ C $ cannot be collinear, unless $ A,B , C $ and $ D $ are all collinear. 

\begin{figure}[t]
\begin{center}
\includegraphics[scale=0.4]{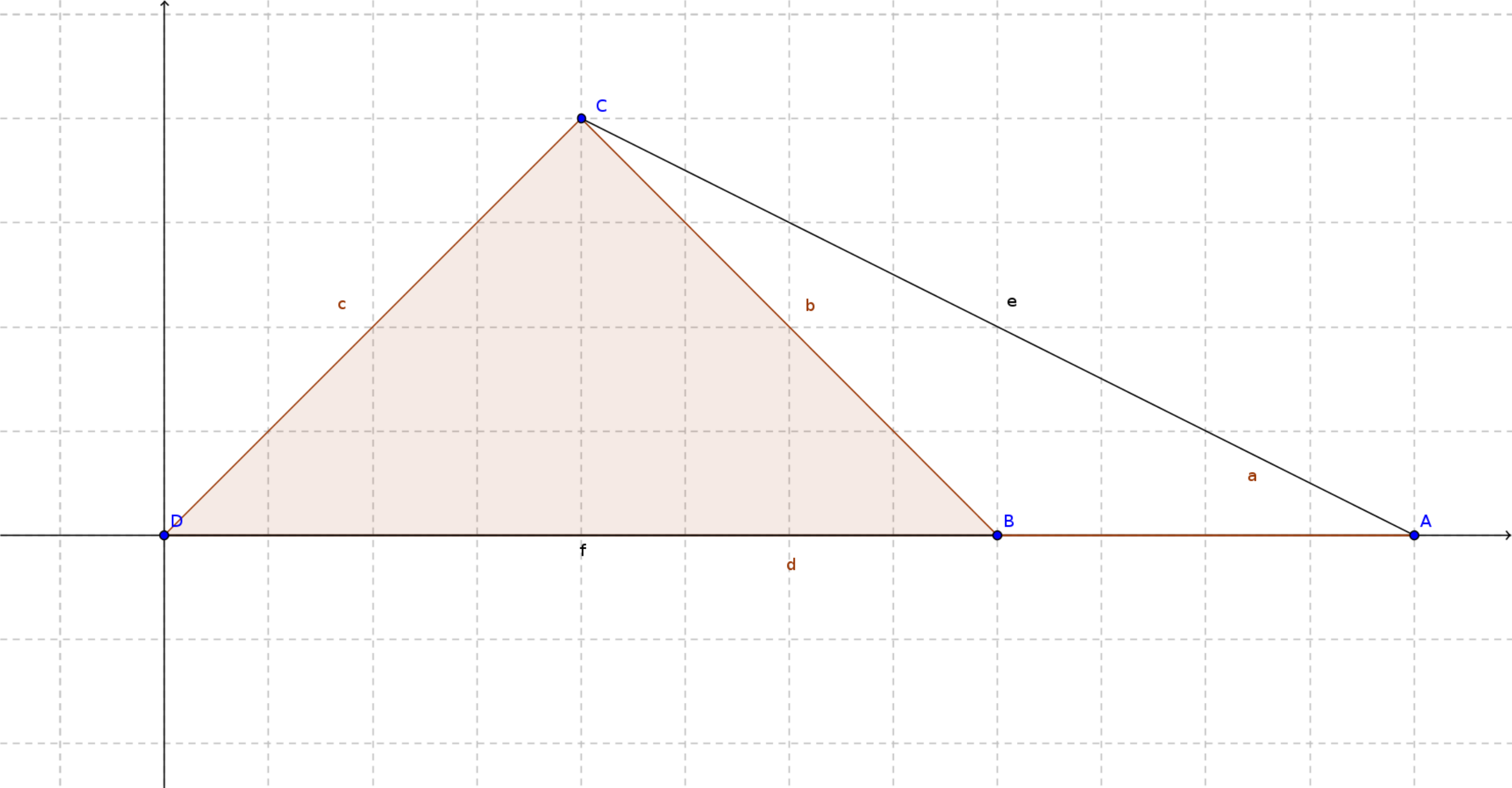}
\caption{Degenerate configuration with $ f = 2 w $\label{fig:degenerate-case-1}} 
\end{center}
\end{figure}

Consider the  ideal $ J_{ABD}=<h _1, h _2, h _3 , h _4, h _5 , R, 
\Delta _{ ABD}- s>$, where $ s $ is a slack variable. Computing the elimination ideal and solving for $ s $ yields 
\[
    s=\pm\frac{ \sqrt{(-d+a+c+b)(d-a+c+b)(d+a-c+b)(d+a+c-b)}(b+c)(b-c)da}{2(ab+cd)(ac+bd)}
\]
Consequently $ \Delta _{ ABD } = s = 0 $  when either $ A, B , C  $ and $ D $ are collinear, or $ b = c $.
In the latter case, $ h _3 = h _5 = 0 $ yield $ f (f - 2 w) = 0$    which, since we assume $ f>0 $,  gives $ f = 2w $, which means that the triangle $ BCD $ is isosceles, see Figure \ref{fig:degenerate-case-1}.
\begin{figure}[t]
\begin{center}
\includegraphics[scale=0.4]{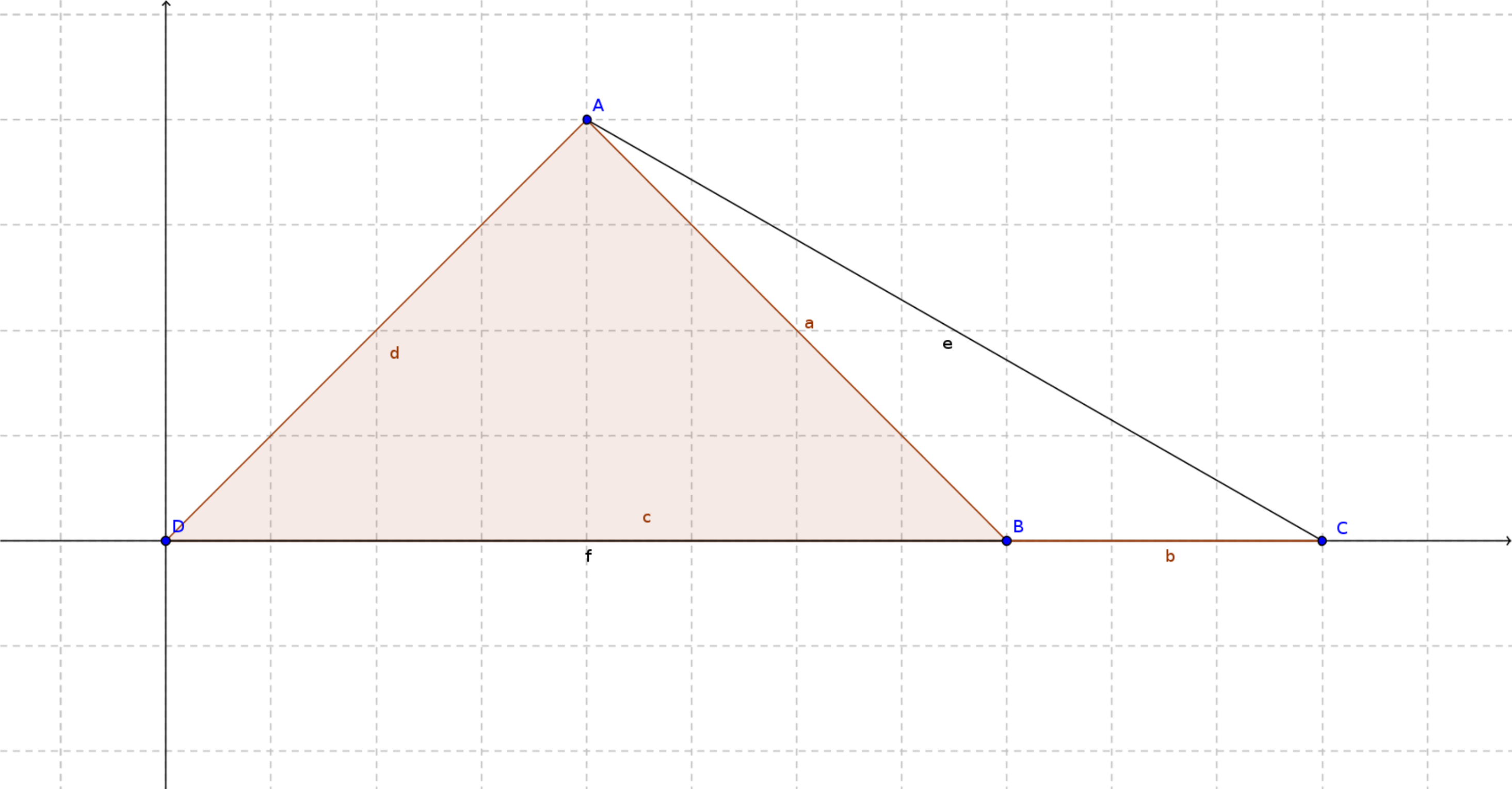}
\caption{Degenerate configuration with $ f = 2 u $\label{fig:degenerate-case-2}} 
\end{center}
\end{figure}

Finally, consider the ideal  $ J_{BCD}=<h _1, h _2, h _3 , h _4, h _5 , R, 
\Delta _{ BCD}- \eta>$, where $ \eta $ is a slack variable. Computing the elimination ideal and solving for $ \eta $ yields 
\[
    \eta =\pm\frac{  \sqrt{(-d+a+c+b)(d-a+c+b)(d+a-c+b)(d+a+c-b)}(a+d)(a-d)cb}{2(ab+cd)(ac+bd)}   
\]
Consequently, $ \Delta _{ BCD } = \eta = 0 $  when either $ A, B , C  $ and $ D $ are collinear, or $ a = d $.
In the latter case, $ h _1 = h _4 = 0 $ yield $ f (f - 2 u) = 0$    which, since we assume $ f>0 $,  gives $ f = 2u $,  which means that the triangle $ ABD $ is isosceles, see Figure \ref{fig:degenerate-case-2}.

Now that we have studied the degenerate cases, we can go back to the non-degenerate ones and conclude the proof.   Assuming that all the $ \Delta $'s are non-zero, we can apply  Theorem \ref{thm:hulls}.
Since $N \geq 0 $, if $ \Delta _{ ABC }  >0 $, then  the convex hull  of the points $ A, B , C, $ and $ D $ is one of  $ABC$, $CAD$, $ABDC$, and $ ADBC $.  If $ \Delta _{ ABC }  <0 $, then  the convex hull  of the points $ A, B , C, $ and $ D $ is one of  $ACB$, $CDA$, $ACDB$, and $ ACBD $. 
This concludes the proof.

\end{proof} 
\begin{figure}[t]
\begin{center}
\includegraphics[scale=0.4]{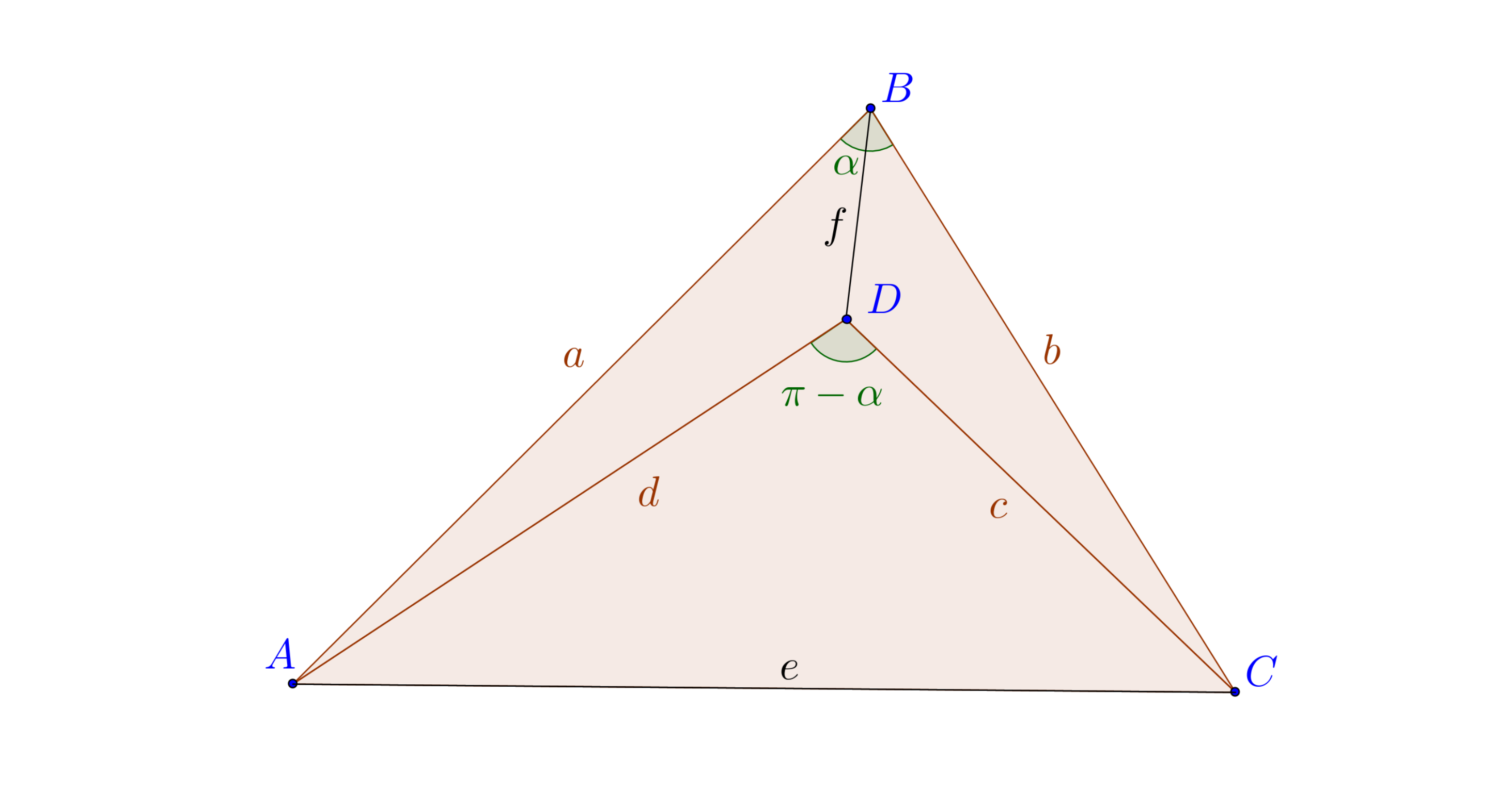}
\caption{Concave quadrilateral $ADCB$  with supplementary angles $\angle ADC $ and $ \angle ABC$. The orange shaded area represents the convex hull $ACB$. \label{fig:supplementary-concave}}
\end{center}
\end{figure}

\begin{remark}[Geometrical Interpretation]
   Note that  the polynomial $ K $, can be interpreted as follows.
We consider the cyclic quadrilateral in Figure \ref{fig:cyclic}.
Since the quadrilateral is cyclic, then  $\measuredangle CDA=\pi-\measuredangle CBA = \alpha $. Hence,   by the law of cosines we have
$ e^2 
= c ^2 + d ^2 - 2 cd \cos \alpha  $ and $ e ^2 = a ^2 + b ^2 + 2 ab \cos \alpha $. Taking a suitable linear combination of these equations  yields the expression
\[
K = e ^2 (ab + cd ) - ab (c ^2 + d ^2) - cd (a ^2 + b ^2)=0.
\]

The exact same computation holds for any of the configurations in Theorem \ref{thm:R=0}. In particular, the computation holds for the cases represented in  Figures \ref{fig:supplementary},\ref{fig:degenerate-case-1},\ref{fig:degenerate-case-2}, and \ref{fig:supplementary-concave}. 
Hence, in all cases,  the expression  $ K = 0 $ is just a linear combination of the cosine law applied to two triangles. 
However, by  Theorem \ref{thm:R=0} we know that in the cyclic case $ P = 0 $, while in the other  cases $ R = 0 $, in which case, whenever the configuration is not cyclic,   $ P >0 $ by the well known Ptolemy's inequality.

Note that, the equation   $ R = 0 $ expresses the Generalized Quadrilateral Theorem (Theorem \ref{thm:generalized-quadrilateral})
for  the configurations mentioned in Theorem \ref{thm:R=0}, and thus, gives the distance between the midpoints of the segments $ AC $ and $ BD $  as 
\[v = \frac{1}{2} \sqrt{ \frac{ (bc + ad) ^2 } { e ^2 }},\]
for those configurations. 
As a consequence, this formula is applicable to specific quadrilaterals having supplementary angles, but is not suitable for cyclic quadrilaterals, as $R \neq 0$ in such cases.  For instance, this formula is valid for both the configuration in Figure  \ref{fig:supplementary} and the concave configuration in  Figure \ref{fig:supplementary-concave}.
 

For planar configurations, if $ R = 0 $ then we also have that $ K =  0 $.  Solving $ K = 0 $ for $ e ^2 $ and substituting the result into $ v $ yields another formula for $ v $ which applies to the configurations specified in Theorem \ref{thm:R=0}  and is independent of the diagonals
\[
    v = \frac{1}{2 } \sqrt{ \frac{(b c + a d) (a b + c d)}  { (ac + bd) }   }.
\]
\end{remark} 

\section{Polynomials  Conditions for Tilted Kites}\label{sec:equalangles}
In the previous section we recognized that the polynomial $ K $ can be interpreted as an expression coming from the law of cosines. There are a few more interesting polynomials that can be obtained that way and are associated to various configurations of four points. 
Recall that tilted kites are quadrilaterals with two opposite equal angles, and they can be convex (see Figure  \ref{fig:equal-angles}) or concave (see Figure \ref{fig:concave-tilted}). 
Convex  tilted kites are  studied in detail in  \cite{josefsson2018properties}. 
\begin{figure}[t]
\begin{center}
\includegraphics[scale=0.4]{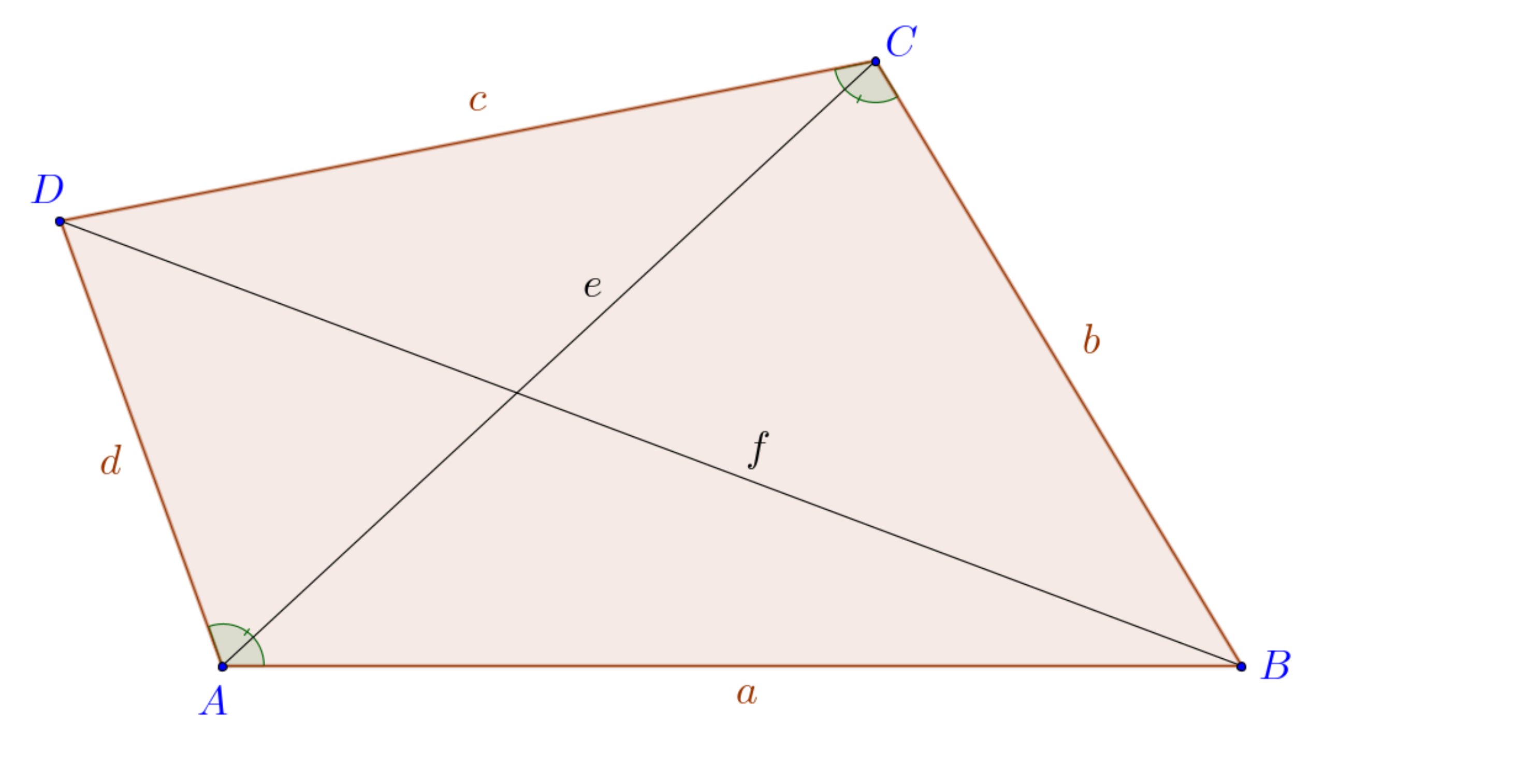}
\caption{Tilted kite with equal angles $\angle BAD $ and $ \angle BCD$.\label{fig:tilted}}
\end{center}
\end{figure}
\begin{figure}[t]
\begin{center}
\includegraphics[scale=0.5]{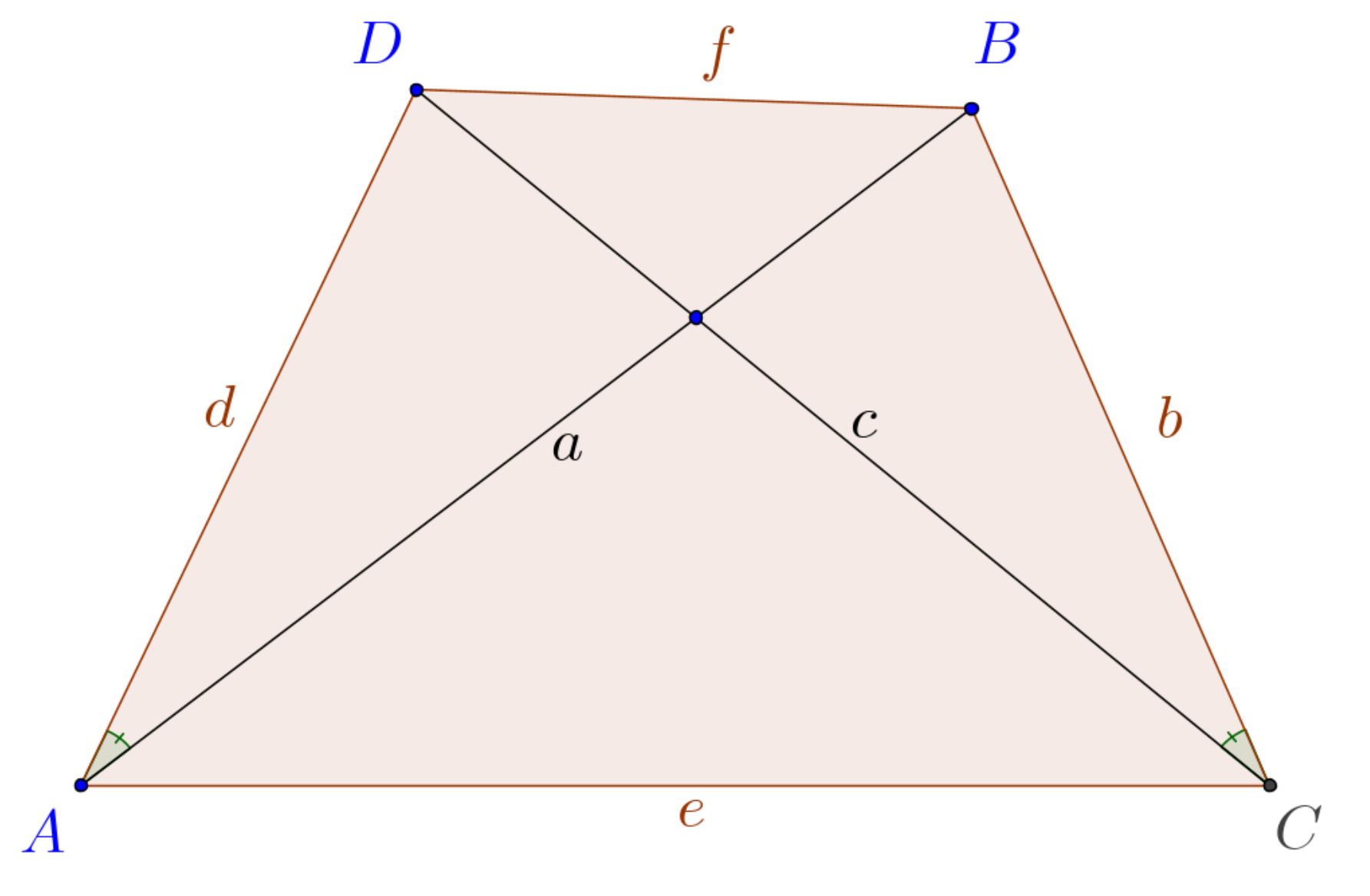}
\caption{Convex quadrilateral $ACBD$ with equal angles $\angle BCD $ and $ \angle BAD$.\label{fig:equal-angles}}
\end{center}
\end{figure}
Consider the  convex tilted kite in Figure \ref{fig:tilted}, the concave tilted kite in Figure \ref{fig:concave-tilted}, and  the quadrilateral \ref{fig:equal-angles}.  Suppose $\measuredangle BAD=\measuredangle BCD = \alpha $. Then,  it follows from  the law of cosines that
$ f^2 
= a ^2 + d ^2 - 2 ad \cos \alpha  $ and $ f ^2 = b ^2 + c ^2 - 2 bc \cos \alpha $. Taking a linear combination of these  yields the expression
\[
K_T = f ^2 (bc - ad ) - bc (a ^2 + d ^2) + ad (b ^2 + c ^2)=0, \label{eqn:KT}
\]
which coincides with the equation in the proof of Theorem 4.2 in \cite{josefsson2018properties}.
Note that the same expression  holds in several cases  including but not limited to  the configurations in Figures \ref{fig:tilted},  \ref{fig:equal-angles},  and  \ref{fig:concave-tilted}.

We also find the polynomials 
\begin{align}
  P_T& = (ac-bd+ef)\label{eqn:PT}\\
  Q _T & = (ac-bd-ef)\label{eqn:QT}\\
  R _T 
       & = (ab-cd)^2 - f ^2 (a ^2 + b ^2 + c ^2 + d ^2 - e ^2 - f ^2) \label{eqn:RT}  
\end{align} 
which are related by the equation
\begin{equation}\label{eqn:identity-T}
    f^2CM   = 2(P _T Q _T R _T - K _T ^2). 
\end{equation}
Moreover, in analogy with the expressions for $ S $  and $ W $,   we also  find the  following polynomials
\begin{align} S_T  = &  f ( b c - a d) - e (c d - a b)   \\
    W _T  = &  -a c (-a ^2 + b ^2 - c ^2 + d ^2 + e ^2 + f ^2) \\
         & - b d (a ^2 - b ^2 + c ^2 - d ^2 + e ^2 + f ^2) + e f ( a ^2 + b ^2 + c ^2 + d ^2 - e ^2 - f ^2) . 
\end{align} 
  Clearly, $ S _T $ is related to the other polynomials by the following equations
\begin{align} 
   - f S _T + K _T -(cd - a b)  P _T & =0\\
   - P _T W _T + S _T ^2+CM/2  & = 0   
   \end{align}

\begin{figure}[t]
\begin{center}
\includegraphics[scale=0.5]{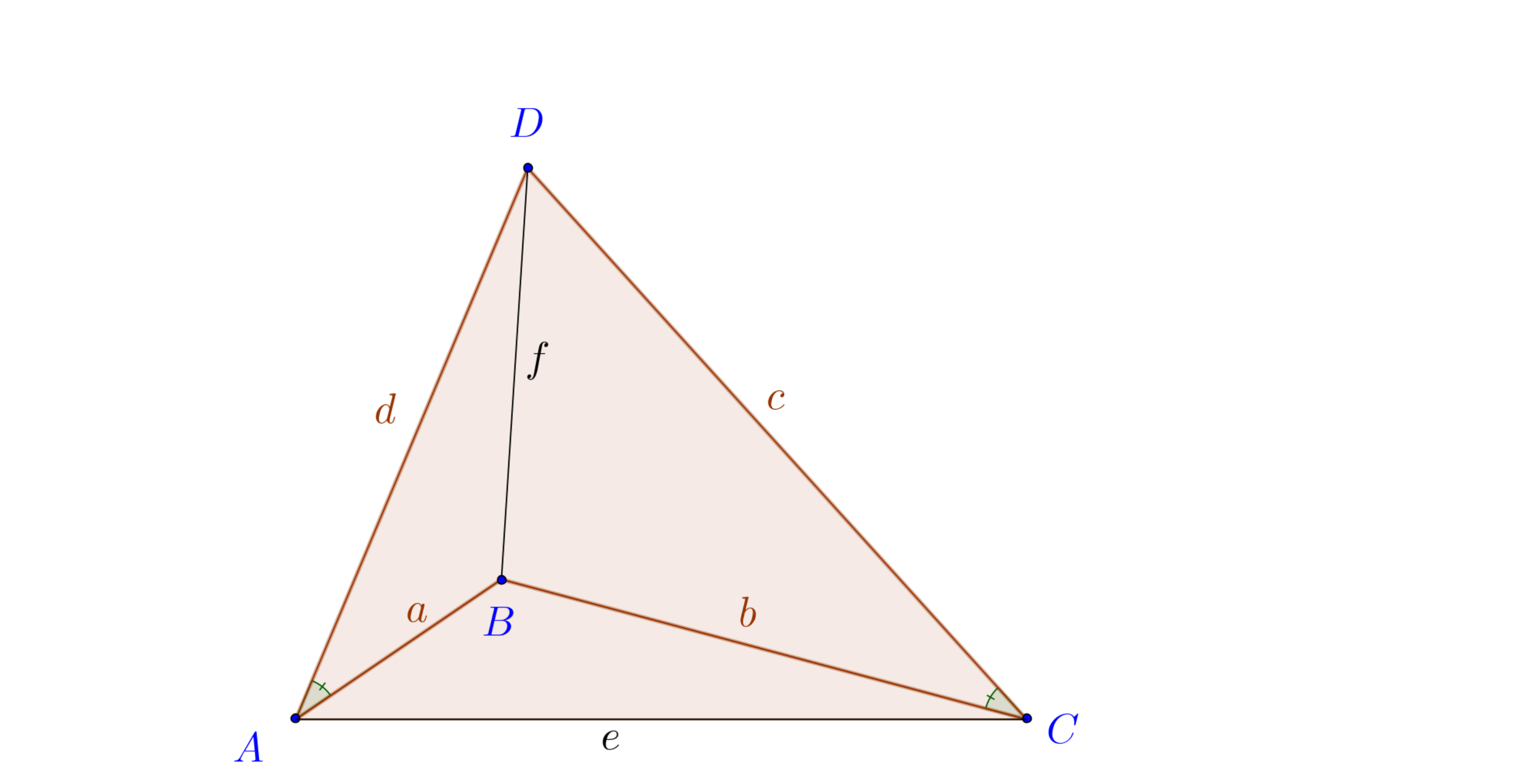}
\caption{Concave tilted kite  $ABCD$ with equal angles $\angle BAD $ and $ \angle BCD$. The orange shaded area denotes the convex hull $ACD$.\label{fig:concave-tilted}}
\end{center}
\end{figure}


Suppose that we have  four point $ A, B , C  $ and $ D $  lying in the plane. Then, $ CM = 0 $, and it is clear from equation \eqref{eqn:identity-T} that
we can consider the following three cases
\begin{enumerate} 
    \item $ Q _T = 0 $: 
 Clearly, by equation  \eqref{eqn:identity-T},  we must have $ K _T = 0 $. Moreover, by the Converse of Ptolemy's Theorem  the four points form a convex cyclic quadrilateral  with counterclockwise order $ ACBD $  (see Figure \ref{fig:equal-angles}) or $ADBC$ , or a straight line.
 Notice that in  Figure \ref{fig:equal-angles} the  angle $\angle DAB$ between the  side  $ d $ and the  diagonal  $ AB$ is equal to the angle between the opposite side and the other diagonal. Hence, by elementary geometry, we have that  $ ABCD $ is  a convex cyclic quadrilateral with diagonals $ a $ and $ c $.
 \item    $ P _T = 0 $: Again, we must have $ K _T = 0 $, but now   the four points  form  a convex cyclic quadrilateral  with counterclockwise order $ ACDB $   or $ABDC$  or a straight line. 
 \item $ R _T = 0 $: As in the previous two cases we must have $ K _T = 0 $.  Looking at Figures \ref{fig:tilted},  and  \ref{fig:concave-tilted}, it is natural to conjecture that  $ R _T = 0 $ is    a necessary condition for a quadrilateral to be a tilted kite. This conjecture is made precise in the following theorem.

 \end{enumerate}

\begin{theorem}\label{thm:RT=0}
     Let $ A, B , C $ and $ D $ be four points in the plane that satisfy  $ R_T = 0 $ then the angles  $\angle BAD  $ and $\angle BCD  $ are equal.  
     Moreover, if  all the $ \Delta  $'s are nonzero then 
     \begin{enumerate}
         \item If  $ \Delta _{ ABC } <0 $ the convex hull  of the points $ A, B , C, $ and $ D $ is one of $ ADCB $,  $ACB$, $ CDA $. 
         \item If  $ \Delta _{ ABC }  >0 $, then  the convex hull is one of  $ABCD$, $ABC$, and $ CAD $. 
     \end{enumerate} 
    In particular, if  $ ad  - bc = 0 $, then the tilted kite is  a parallelogram or a rhombus. 
    Finally, If at least one of the $ \Delta  $'s is zero then we have the following degenerate cases
     \begin{enumerate}
         \item The points $ A,B,C $ and $ D $ are collinear.  
         \item The points $ A, B  $ and $ c  $ are collinear and $ c = d $. 
         \item The points $ A , C $ and $ D $ are collinear and $ a = b $ . 
     \end{enumerate}  
 \end{theorem} 
\begin{proof}

  Assume $ R_T = 0 $.  Since the configuration is planar $ CM = 0 $, and thus,  by equation \eqref{eqn:identity-T}, we also have that   $ K_T = 0 $.
  The equation $ K_T = 0 $ can also be written as 
\[
    bc \left[ f ^2 - (a ^2 + d ^2) \right] = ad \left[ f ^2 - (b ^2 + c ^2) \right].   
\]
Applying the law of cosines we have $ f ^2 = a ^2 + d ^2 - 2 ad \cos \alpha $ and $ f ^2 = b ^2 + c ^2 - 2 bc\cos \beta $, where $\measuredangle BAD = \alpha $ and $ \measuredangle BCD = \beta $. Hence we obtain the equations 
\[ bc \left[ f ^2 - (a ^2 + d ^2) \right] = - 2 abcd \cos \alpha, \quad  ad \left[ f ^2 - (b ^2 + c ^2) \right]= - 2abcd \cos \beta . \]
Hence, $ K_T = 0 $ implies that $ \cos \alpha =  \cos \beta $, which can be also written as 
\[
    \cos \alpha - \cos \beta =- 2 \sin \left( \frac{ \alpha + \beta } { 2 } \right) \sin \left( \frac{ \alpha - \beta } { 2 } \right)=0,
\]
which means that $  \alpha + \beta = 2\pi n  $ or $\alpha - \beta  = 2 \pi m  $ with $ m, n = 0 , \pm 1 , \pm 2, \ldots $. Since $ \alpha , \beta \leq \pi $ the only possible solution  is $ \alpha  = \beta  $.  This proves that 
the angles  $\angle BAD  $ and $\angle BCD  $ are equal.

To complete the proof choose a  coordinate system such that $ A=(0,0) $, $ B = (u,v) $, $ C = (e, 0) $ and  $ D = (w,z) $. Then we have the following polynomials
\begin{align*} 
g _1 & = u^2 + v ^2 - a ^2 \\
g _2 & = (e-u) ^2 + v ^2 - b ^2\\
g _3 & = (w -e)^2 + z ^2 - c ^2 \\
g _4 & = w ^2 + z ^2 - d ^2 \\
g _5 & = (u-w)^2+(z-v)^2-f^2
\end{align*} 
We show that $ M >0 $. Consider the ideal $ J_M=<h _1, h _2, h _3 , h _4, h _5 , R, M - t>$, where $ t $ is a slack variable. 
Calculating the elimination ideal of $ J_M $  with Maple,  eliminating the variables $ e,f, u,v, w, z $ yields a polynomial in $ a,b,c,d $ and $ t $. Solving  for $t$ yields
\[
     t = - \frac{ abcd\, \Gamma }{4(ad-bc)^2}
\]
where $ \Gamma =(d+a+c+b)(d+a-c-b)(-d+a+c-b)(-d+a-c+b) $. 
Repeating this calculation for the ideal $< h _1 , h _2 , h _3 , h _4 , h _5 , R, \Delta _{ ABD}-s>$ yields 
\[
   s= \pm\frac{  \sqrt{- \Gamma }\,da}{(2(ad-bc))},
\]
which means that, if $ad-bc \neq 0 $,  we must have $ - \Gamma \geq 0 $, 
otherwise $ \Delta _{ ABD } =s $ is not  a real number. 
This shows that, provided that $ ad -  b c \neq 0 $, we have that $ M  \geq 0 $. If all the $ \Delta $'s are non-zero,  then $ M >0 $ and there are the following possible cases: $ ABCD$, $ABC$ and $ CAD $  if $ \Delta _{ ABC}>0 $  
and $ ADCB $, $ACB$ and $CDA$ if $ \Delta _{ ABC } <0 $. The quadrilaterals $ ABCD $ and $ ADCB $ are convex tilted kites (see Figure \ref{fig:tilted}), while  $ACB$  $CDA$,  $ ABC $ and $ CAD $ are concave tilted kites (see Figure \ref{fig:concave-tilted}).
Suppose $ ad - bc = 0 $.
A computation with Groebner bases shows that 
\[ b^2c^2(a-c)^2(a+c)^2(a-b)^2(a+b)^2\in \sqrt{J},\]
with $ J =< h _1 , \ldots , h _5 , R _T , a d - b c> $.  Therefore, $ c = a $ or $ a = b $. If $ c = a $ then we also have that $ d = b $. If $ a = b $, instead we have that $ c = d$.  In both  cases $ R _T = -f^2(2a^2+2b^2-e^2-f^2)$, and thus we must have $ 2a^2+2b^2-e^2-f^2=0$, which is the parallelogram law. Since the parallelogram law holds, the quadrilateral is, in both cases,  a parallelogram. Note that, in the first case we have shown that opposite sides are equal, while in the second case adjacent sides are equal. Therefore, in the second case the parallelogram is actually a rhombus.  

Now, we consider the case where one or more of the $ \Delta $'s is zero.  First of all we have that 
\[
    \Delta _{ BCD } = \pm\frac{  \sqrt{- \Gamma }\,cb}{(2(ad-bc))},\quad 
        \Delta _{ ABC } = \pm\frac{  \sqrt{- \Gamma }\,(c+d)(c-d)ba}{2(ac-bd))(ad-bc)}\]
and 
\[
        \Delta _{ ACD } = \pm\frac{  \sqrt{- \Gamma }\,(a+b)(a-b)dc}{2(ac-bd))(ad-bc)}.
\]

Therefore, $ \Delta _{ BCD }  =0$ or $ \Delta _{ ABD } = 0 $ if and only if $ \Gamma = 0  $. Consequently, if  $ \Delta _{ BCD }  =0$ or $ \Delta _{ ABD } = 0 $, then all the $ \Delta $'s must be zero, and the four points lie on a straight line. 

If  $ c = d $ then $ \Delta _{ ABC } = 0$, so that the points $ A $, $ B $ and $ C $ lie on the same line. From the equations $  g _3 = g _4 = 0 $ we find that $ e (e - 2 w) = 0 $, that is, $ e = 2w $. 
This configuration can be described as an isosceles triangle $ ABC $, with the point $ B $ on the segment $   AC $.
If  $ a = b $ then $ \Delta _{ ACD } = 0$, so that the points $ A $, $ C $ and $ D $ lie on the same line. From the equations $  g _3 = g _4 = 0 $ we find that $ e (e - 2 u) = 0 $, that is, $ e = 2u $. 
This configuration can be described as an isosceles triangle $ ACD $, with the point $ D $ on the segment $  AC$. 

\end{proof}

\begin{remark} 
Notice that, by the previous theorems, for the tilted kites in Figures  \ref{fig:tilted},   and  \ref{fig:concave-tilted}  we must  have that $ R _T = 0 $ and $ K _T = 0 $.  
For the   quadrilateral in Figure \ref{fig:equal-angles}, however, the  angle $\angle DAB$ between the  side  $ d $ and the  diagonal  $ AB$ is equal to the angle between the opposite side and the other diagonal. Hence, in this case, by elementary geometry, we have that   the vertices $ ABCD $  define  a convex cyclic quadrilateral with diagonals $ a $ and $ c $.
By Ptolemy's theorem we then have that $ Q _T = 0 $, and by the cosine law we have that $ K _T = 0 $.  
In this case we have $ K _T = Q _T = 0 $, instead. 
\end{remark}

We conclude this section with a theorem that shows that the quadrilateral determined by the condition $ Q _T P _T = 0 $, and the quadrilateral defined by the condition $ \tilde R _T = 0 $ (that is the condition $ R _T = 0 $ but with different distances) are related to each other by a reflection. 
To introduce this theorem recall that $ A $, $ B  $, $ C  $, and $ D   $ are  four points on the plane, and, as usual,   $ a = |AB| $,  $ b = |BC| $, $ c = |CD|$, $ d = |DA| $, $ e = |AC| $, $ f = |BD| $, are the Euclidean distances between pair of points.
Let $ \tilde C $ be the reflection of the point $ C $ over the line through the points $B$ and $D$. 
Let $\tilde b =  |B\tilde C| $, $ c = \tilde |\tilde CD|$, and  $ \tilde e = |A\tilde C| $, be the distance between $ \tilde  C $ and   the points $ B, D $ and $ A $, respectively. Clearly, $ \tilde b = b $ and $ \tilde c = c $. Let $ K_T, P _T , Q _T  $ and $ R _T $ be the  polynomials given in equations (\ref{eqn:KT}), \ref{eqn:PT}),(\ref{eqn:QT}) and (\ref{eqn:RT}), respectively. Let $ \tilde K _T , \tilde P _T, \tilde Q _T $ and $ \tilde R _T $  be the analogue of   $ K_T, P _T , Q _T  $ and $ R _T $ for the points $ A, B , \tilde C $ and $ D $. It is then easy to see that $ K _T = \tilde K _T $, while $  \tilde P _T, \tilde Q _T $ and $ \tilde R _T $ are obtained from $  P _T , Q _T  $ and $ R _T $  by replacing $ e $ with $ \tilde e $.
Moreover we denote by $ CM $ the Cayley-Menger determinants for the points   $ A $, $ B  $, $ C  $, and $ D   $  and $ \widetilde{CM} $ the Cayley-Menger determinants for the points   $ A $, $ B  $, $\tilde C  $, and $ D   $.

Using these notations we prove the following theorem 
\begin{theorem}
      Let $ A, B, C  $ and $ D $ be four points that lie on the plane, and let $ \tilde C $ be the reflection of the point $ C $  over the line through the points $ B $ and $ D $. Then  $ \tilde R _T = 0 $, if and only if  $ P _T Q _T = 0 $. 
\end{theorem} 
\begin{proof} 
This theorem can be proved by an application of Theorem \ref{prop:RadicalMembership}. Here we give a proof that does not rely on Groebner bases computations.

Suppose $ P _T = 0 $. 
Substituting $ P _T = 0 $ into equation \eqref{eqn:identity-T} yields 
\[
    CM = -2 \frac{ K _T ^2 } { f ^2 } = 0,
\]
and thus $ K _T = 0 $. 
Equation  \eqref{eqn:identity-T} for the points $ A, B , \tilde C $ and $ D $ is   
\begin{equation} \label{eqn:widetildeCM}
    \widetilde{CM} = 2 \frac{ \tilde P _T \tilde Q _T \tilde R _T - K _T ^2 } { f ^2 } .
\end{equation} 
Since 
$ K _T = 0 $ we find that 
\[\widetilde{CM} = 2 \frac{ \tilde P _T \tilde Q _T \tilde R _T } { f ^2 }. \]
Therefore, we must have $  \tilde P _T \tilde Q _T \tilde R _T=0 $, since $ \widetilde{CM } = 0 $ for a planar configuration. 
If $ \tilde P_T  = 0 $ we must have $e = \tilde e$, which only happens when either 
$ A $ or  $ C $ lie on the  line through $ B $ and $ D $. However, since we are assuming that  $ P _T = 0 $  the  points  $ A, B, C $ and $ D $ are either cocircular or collinear. Since three collinear points cannot be cocircular  then all the points must be collinear. In this case $ \tilde Q _T= \tilde P _T= \tilde R _T = 0 $. 
If, on the other hand, $ \tilde Q _T = 0 $ then $ P _T - \tilde Q _T  = (e + \tilde e) f = 0 $. But this is impossible since  we take $ e, f>0 $.  

The analysis of the case $ Q _T = 0 $, is analogous to the case $ P _T = 0 $, and hence it will be omitted. 

Now, suppose $\tilde R _T = 0 $. 
Substituting $ \tilde R_T = 0 $ into \eqref{eqn:widetildeCM} yields 
\[
      \widetilde{CM} = -2 \frac{ K _T ^2 } { f ^2 } = 0,
\]
and hence $ K _T = 0 $.
Since, by equation \eqref{eqn:identity-T},  $ CM = 2 \frac{ P _T Q _T R _T - K _T ^2 } { f ^2 } $ it follows that 
\[CM = 2 \frac{ P _T Q _T R _T } { f ^2 }. \]
Since the configuration is planar, $ CM = 0 $ and thus $ P _T Q _T R _T = 0 $. If $ R _T  \neq 0 $, then we are done. It remains to see what happens when  $ R _T = 0 $. Since $ \tilde R _T =0 $, we have that $ R _T = 0 $ if and only if $ e = \tilde e $. This last condition   holds if and only if either $ A $ or $ C $ lie on the line through the points $ B $ and $ D $. Note that, according to Theorem \ref{thm:RT=0}, if $ R _T = 0 $ and $ \Delta _{ ABD } = 0 $ or $ \Delta _{  BCD } = 0 $, then all the four points must lie on a line.  Therefore, in this case,  by the  Ptolemy's theorem,  the points $ A, B , \tilde C $ and $ D $ also  satisfy the following equalities: $ P _T = Q _T = 0 $. This concludes the proof. 

\end{proof} 
\section{Finding Additional Polynomial conditions}
Because these conditions are intrinsically interesting, and because of the potential applications to Celestial Mechanics and other fields, we think it is advantageous to find additional polynomial conditions. 
Two examples are provided here to illustrate how you can find new conditions. 

First we consider a very simple minded approach. 
If we apply the cosine law to find an expression for $ f ^2 $ instead of $  e ^2 $ then we can obtain the following polynomial
\begin{align*} 
      K_1 =(ad + bc)f ^2 - (a ^2 + d ^2) bc - (b ^2 + c ^2) ad.
\end{align*} 
We can also consider 
\[
      R_1  =(ab+cd)^2-f^2(a^2+b^2+c^2+d^2-e^2-f^2).
\]
These polynomials satisfy the following conditions
\begin{align*} 
    2(K_1^2-PQR_1)+f^2CM& =0\\
    fS +K_1+(ab+cd)P& =0.
\end{align*} 
These polynomials relationships are very similar to the ones found by Pech \cite{10.1007/978-3-642-21898-9_34}, and therefore it is not worthwhile to study them. 
However, combining  these equations with  \eqref{eqn:identity2} and \eqref{eqn:identity3} we can  obtain the following symmetrized relations
\begin{align*} 
 (f-e)S+K+K_1+(ab+cd+bc+ad)P& =0\\
 2[K^2+K_1^2-PQ(R+R_1)]+(e^2+f^2)CM& =0
\end{align*}
which may quench some scholar's thirst for symmetry. 

Another  less trivial approach consists in using  Groebner bases type computations to find new polynomials and new relations. This approach seems more promising. 
For instance, by using  Maple to compute a Groebner basis for the ideal generated by  $ h _1 , \ldots , h _5 $ and $ P $ with respect to a graded reverse lexicographic order we obtain several polynomials, one of which is 
\[
    (a^2ce-ac^2f+ad^2f-ae^2f-cd^2e+cef^2)^2
\]
We can then define 
\[K _G= (a^2ce-ac^2f+ad^2f-ae^2f-cd^2e+cef^2)=  (af-ce)d^2+ce(a^2+f^2)-(c^2+e^2)af.\]
Then
\[
    \frac{ CM d ^2 } { 2 } = PQ _G R _G - K _G ^2 
\]
where  $ Q _G = -ac+bd+ef $ and 
\[
    R _G = (a^2-b^2+c^2-d ^2 +e^2+f^2)d^2- (ae - cf) ^2 
\]
These polynomials are very similar to the ones we studied in section \ref{sec:equalangles} except here  $ b $ and $ d $ are the diagonals. As  a consequence we will not study them because they do not lead to  anything substantially new. 

\section*{Acknowledgements}
The author wish to thank Martin Josefsson  for his helpful comments and suggestions. 

\bibliographystyle{plain}

\bibliography{references}

\end{document}